\documentclass[11pt]{amsart}

\usepackage{enumerate}
\usepackage{amsmath}
\usepackage{verbatim}

\newcommand{\NN}{\mathbb{N}}
\newcommand{\ZZ}{\mathbb{Z}}
\newcommand{\RR}{\mathbb{R}}

\newcommand{\FFF}{\mathcal{F}}

\newcommand{\GGG}{\mathcal{G}}
\newcommand{\htop}{h_\mathrm{top}}

\newtheorem{definition}{Definition}[section]
\newtheorem{theorem}{Theorem}[section]
\newtheorem{lemma}[theorem]{Lemma}

\newtheorem{thma}{Theorem}

\theoremstyle{remark}
\newtheorem*{remark}{Remark}

\DeclareMathOperator{\It}{It}
\DeclareMathOperator{\sign}{sign}

\numberwithin{equation}{section}

\begin{document}

\title[Generalized $\beta$-transformations]{Generalized $\beta$-transformations and the entropy of unimodal maps}

\author{Daniel J. Thompson}

\address{Department of Mathematics, The Ohio State University, 100 Math Tower, 231 West 18th Avenue, Columbus, Ohio 43210}

\email{thompson@math.osu.edu}
\date{\today}
\thanks{This work is supported by NSF grant DMS-$1461163$}
\begin{abstract}

Generalized $\beta$-transformations are the class of piecewise continuous interval maps given by taking the $\beta$-transformation $x \mapsto \beta x ~\pmod 1$, where $\beta>1$, and replacing some of the branches with branches of constant negative slope. If the orbit of $1$ is finite, then the map is Markov, and we call $\beta$ (which must be an algebraic number) a \emph{generalized Parry number}. We show that the Galois conjugates of such $\beta$ have modulus less than $2$, and the modulus is bounded away from $2$ apart from the exceptional case of conjugates lying on the real line.  We give a characterization of the closure of all these Galois conjugates, and show that this set is path connected. Our approach is based on an analysis of Solomyak for the case of $\beta$-transformations. One motivation for this work is that the entropy of a post-critically finite (PCF) unimodal map is the logarithm of a generalized Parry number. Thus, our results give a mild restriction on the set of entropies that can be attained by PCF unimodal maps.
\end{abstract}
\maketitle

\section{Introduction}
For a continuous post-critically finite (PCF) interval map, the exponential of the topological entropy, denoted $\exp h$, is a Perron number. Thurston showed that all Perron numbers can be obtained this way \cite{wT14}. However, for PCF multimodal maps with restricted degree, the situation changes dramatically: the dynamics impose complicated restrictions on which Perron numbers can be attained as $\exp h$. We want to describe these numbers 
by understanding the restrictions on the Galois conjugates of $\exp h$. 

This problem was raised in Thurston's final paper \cite{wT14}, which includes a figure of the set of complex numbers which are Galois conjugates of $\exp h$ for PCF unimodal maps. We denote this set by $\Omega_T$. An ongoing problem raised by Thurston's final paper is to understand the structure of $\Omega_T$. So far, progress has been made by Calegari, Koch and Walker on understanding the region of $\Omega_T$ which lies inside the unit disk \cite{CKW}, and Tiozzo \cite{tiozzo} has shown that the set $\overline{ \Omega_T}$ is path connected. There has been no further progress on understanding the region of $\Omega_T$ which lies outside the unit disk. 

In this paper, we study an analogous problem for a class of interval maps called generalized $\beta$-transformations, and use this to gain at least some information about the outer boundary of $\Omega_T$. 

Our approach is based on the formalism of generalized $\beta$-transformations, as introduced by G\'ora \cite{pG07}. 
The generalized $\beta$-transformations are the class of piecewise continuous interval maps given by replacing some of the branches of a $\beta$-transformation with branches of constant slope $-\beta$.  
We call a generalized $\beta$-transformation post-critically finite (PCF) if the point $1$ has a finite orbit, to unify terminology with the case of continuous multimodal maps. We let $\Omega$ denote the set of Galois conjugates of all $\beta$ such that there exists a PCF generalized $\beta$-transformation.

 The class of generalized $\beta$-transformations contains all $\beta$-transformations, and many continuous interval maps. Of particular interest is the case where $\beta \in(1,2)$, the first branch is increasing, and the second branch is decreasing. This gives a class of continuous unimodal maps among which the entropy of every PCF unimodal map is represented (i.e. every PCF unimodal map is semi-conjugate to a PCF map in this class with the same entropy). In particular, we can conclude that $\Omega_T \subset \Omega$, see \S\ref{unimodal} for more details.




We show that $\Omega$ lies inside a disk of radius $2$. 
Since $\exp h$ and the degree of a (continuous) generalized $\beta$-transformation can be arbitrarily large, this result contrasts sharply with the result that any Perron number can be achieved as the entropy of a general PCF multimodal map. 

Although it is trivial that $\Omega_T$ lies inside a disk of radius $2$, since unimodal maps satisfy $\exp h \leq 2$, our results can be improved by excluding the exceptional case of real-valued Galois conjugates.  Experimental investigation of Beaucoup, Borwein, Boyd and Pinner \cite{BBBP} suggests that a sharp bound should be less than $1.6$.  We show rigorously that all non-real-valued Galois conjugates have modulus bounded away from $2$, at least establishing the principle that $\Omega_T \setminus \RR$ lies in a disk of radius less than $2$.

Our techniques are inspired by an analysis of Solomyak \cite{So94}. For `PCF'  $\beta$-transformations, i.e. those for which the point $1$ has a finite orbit, in which case we call $\beta$ a \emph{Parry number}, Solomyak showed that all Galois conjugates are bounded in modulus by the golden mean, and that this bound is sharp, improving on a bound of $2$ obtained by Parry \cite{Pa}. Furthermore, Solomyak established a Structure Theorem which gives a rather explicit characterization of the largest Galois conjugate in each direction (i.e. with a prescribed argument). Our approach is based on extending these results as far as possible to the setting of PCF generalized $\beta$-transformations.  In particular, we have a version of the Solomyak Structure Theorem for generalized $\beta$-transformations. This result provides an analytic tool for studying the largest modulus of points in $\Omega$, and thus for bounding above the largest modulus of points in $\Omega_T$. Showing that this theory applies in the context of $\Omega_T$ is one of the main points of this article, as no techniques were previously available for attacking this problem.

The main idea of the argument is to characterize those elements of $\Omega$ with $|z|>1$ as the inverse of a zero of an analytic function in the class $
\FFF = \{ T(w) = 1+ \sum_{j=1}^{\infty} a_j w^{j} : a_j \in [-1,1]\}.$
This correspondence is obtained from the expression for the generalized $\beta$-expansion of $1$, and makes the problem tractable to further analysis. We also show that if $\lambda$ is a zero of a function in $\FFF$, then $\lambda^{-1} \in \overline \Omega$. We use this to show that $\overline \Omega$ is path connected.

In \S\ref{sec2}, we introduce generalized $\beta$-transformations. In \S\ref{sec3}, we introduce generalized Parry numbers and generalized Parry polynomials, 
and obtain our basic bound on the size of  $\Omega$.  In \S \ref{sec4}, we establish a description of the outer boundary of $\Omega$ and establish slightly improved bounds on $\Omega \setminus \RR$. In \S \ref{sec3.1}, we study $\overline \Omega$. In \S \ref{unimodal}, we apply our results to unimodal maps.



\section{Generalized $\beta$-transformations} \label{sec2}
The $\beta$-transformations are the class of piecewise continuous interval maps $x \to \beta x \pmod 1$, where $\beta>1$. The class of generalized $\beta$-transformations, introduced by G\'ora \cite{pG07}, are obtained from the $\beta$-transformations by flipping some of the branches so the slope is $-\beta$, and extending the map to a piecewise continuous map of the closed interval $[0,1]$. It is clear what it means to flip a full branch of the map. If we flip the rightmost branch, which is the only branch that is not full, we mean that this `flipped branch' of the map decreases from $1$ to $1-\{\beta\}$, where $\{\beta\}$ is the fractional part of $\beta$.  The precise definition is given below. 
We record the configuration of positive and negative slopes by a vector $E$ of $1$'s and $-1$'s. The $1$'s correspond to increasing branches, and the $-1$'s correspond to decreasing branches.  There is a large literature on using classes of interval maps to give expansions of real numbers \cite{wP64b, DK, DHK, St13, R16}. Generalized $\beta$-transformations were introduced in this context. Connections with the theory of tilings are given in \cite{PR}.
\subsection*{The $(1,-1)$ case}
A case of particular interest in this study is the sign configuration $E=(1,-1)$, in which case the map is a continuous unimodal map. 
For ease of exposition, we define the map rigorously in this case first. 
We let $\beta \in (1,2]$, and we let $I_0= [0, 1/\beta]$ and $I_1 = (1/\beta, 1]$, so that the sets $I_0, I_1$ denote the partition of $[0,1]$ into the two intervals of monotonicity.
In this case, the generalized $\beta$-transformation has the formula

$$
f(x) =
\begin{cases}
\beta x & \text{if }x \in I_0 \\
2- \beta x & \text{if } x \in I_1
\end{cases}
$$
For $j \geq 1 $, we let
$$
d(x, j) =
\begin{cases}
0 & \text{if } f^{j-1}x \in I_0 \\
2 & \text{if } f^{j-1}x \in I_1.
\end{cases}
$$ 
The symbols $d(x, j)$ are the digits used in the generalized $\beta$-expansion of $x$.
For $j \geq 1 $, we let
$$
e(x, j) =
\begin{cases}
1 & \text{if } f^{j-1}x \in I_0 \\
-1 & \text{if } f^{j-1}x \in I_1
\end{cases}
$$ 
We define the `cumulative sign' by $s(x, 1) =1$,
\[
s(x, j+1) = e(x, j) s(x, j) = \prod_{l=1}^{j} e(x, l), 
\]
and we let $s(j) := s(1, j)$. The generalized $\beta$-expansion of $x$ is the expression
\[
x = \frac{s(x,1) d(x, 1)}{\beta} + \frac{s(x,2) d(x, 2)}{\beta^2} + \cdots + \frac{s(x,j) d(x, j)}{\beta^j} + \cdots
\]
By \cite[Corollary 2]{pG07}, this expression is valid for every $x\in [0,1]$. 
\subsection*{All generalized $\beta$-transformations} Let $m \in \NN$ and $\beta \in (m, m+1]$. For such $\beta$, a generalized $\beta$-transformation has $m+1$ branches. We let $E=(E(0), E(1), \ldots, E(m)) \in \{1,-1\}^{m+1}$ be the vector which describes the configuration of slopes of the map (where an entry $1$ corresponds to positive slope, and an entry $-1$ corresponds to negative slope). We partition $I$ into $m+1$ intervals
\[
I_0 = \left[0, \frac{1}{\beta} \right ], I_1 = \left(\frac{1}{\beta}, \frac{2}{\beta} \right ], \ldots, I_m = \left(\frac{m}{\beta}, 1 \right ],
\]
and we define the \emph{$(\beta, E)$-transformation} $f=f_{\beta, E}$ by the formula
$$
f(x) =
\begin{cases}
\beta x -k& \text{if }x \in I_k \text{ and } E(k)=1 \\
- \beta x + k+1 & \text{if } x \in I_k \text{ and } E(k) =-1.
\end{cases}
$$
Note that the intervals $I_j$ are defined to include their right end-points, and $f$ is defined on the whole interval $[0,1]$. In the case that all entries of $E$ are $1$, then $f$ is an extension of the classical $\beta$-transformation $x \to \beta x \pmod 1$ to a piecewise continuous map of the closed interval $[0,1]$.

For $j \geq 1$, we let
$$
d(x, j) =
\begin{cases}
k & \text{if } f^{j-1}x \in I_k \text{ and } E(k)=1\\
k+1 & \text{if } f^{j-1}x \in I_k \text{ and } E(k)=-1.
\end{cases}
$$ 
For $j \geq 1 $, we let
$
e(x, j) = E(k) \text{ if } f^{j-1}x\in I_k.
$ 
We define the `cumulative sign' by $s(x, 1) =1$,
\[
s(x, j+1) = e(x, j) s(x, j) = \prod_{l=1}^{j} e(x, l).
\]
Again, G\'ora shows that for every $x\in [0,1]$,
\begin{equation} \label{betaexp}
x = \frac{s(x,1) d(x, 1)}{\beta} + \frac{s(x,2) d(x, 2)}{\beta^2} + \cdots + \frac{s(x,j) d(x, j)}{\beta^j} + \cdots
\end{equation}
We refer to this expression as the $(\beta, E)$-expansion for $x$. For the $(\beta, E)$-expansion of $1$, we write 
$d(j):=d(1, j)$ and $s(j):=s(1, j)$.

We sometimes write $(\beta,E)$-expansions using sequence notation   
\[
((s(x,1), d(x,1)), (s(x, 2), d(x, 2)), (s(x, 3), d(x, 3)), \ldots).
\]
The set-up above includes the classic $\beta$-expansion simply by setting all entries in $E$ to be $1$. In this case, $s(x,j) =1$ for all $x$ and $j$, and \eqref{betaexp} reduces to the standard $\beta$-expansion of R\'enyi and Parry.  

\subsection{Finite versus infinite $(\beta, E)$-expansions}
It is possible in the definition of the $(\beta, E)$-expansion that there exists $n$ so that $d(x,j)=0$ for all $j>n$, and thus the $(\beta, E)$-expansion of $x$ is finite. This can only happen if $E(0)=1$, so $0$ is a fixed point, and $f^{n}x=0$. Since the set of preimages of $0$ is a subset of the left end-points of the intervals $I_j$, we must have $f^{n-1}x \in\{1/\beta, 2/\beta, \ldots, [\beta]/\beta\}$. 

We explain how to derive an infinite $(\beta, E)$-expansion from a finite $(\beta,E)$-expansion. We start with the case that $x=1$ has a finite $(\beta, E)$-expansion.  The finite $(\beta, E)$-expansion of $1$ is thus
\begin{equation} \label{eqn:1}
1 = \frac{s(1)d(1)}{\beta}+ \frac{s(2)d(2)}{\beta^2} + \ldots + \frac{s(n)d (n)}{\beta^n}, 
\end{equation}
where 
$d(n) \neq 0$. Let $d'(j)=d(j)$ for $j\in \{1, \ldots, n-1\}$, and 
\[
d'(n)=
\begin{cases}
 d(n)-1 &\text{ if } s(n)=1 \\
d(n)+1 &\text{ if } s(n)=-1,
\end{cases}
\]
noting that $0\leq d'(n) \leq d(1)$. This is because $1\leq d(n)$, and $d(n)\leq d(1)$. It is easily checked that the only way we can have $d(n)= d(1)$ is if $\beta \in \mathbb N$, and $f(1)=0$. In this case, since $s(1)=1$, $d'(n)= d'(1)=d(1)-1$. 

We have
\begin{equation} \label{eqn:3}
1 = \frac{s(1)d'(1)}{\beta}+ \frac{s(2)d'(2)}{\beta^2} + \ldots + \frac{s(n)d'(n)}{\beta^n} + \frac{1}{\beta^n},
\end{equation}
and thus for any $m\geq0$,
 \begin{equation}\label{eqn:2}
\frac{1}{\beta^{mn}} = \sum_{j=1}^n \frac{s(j)d'(j)}{\beta^{j+mn}} + \frac{1}{\beta^{(m+1)n}}.
\end{equation}
Note that by \eqref{eqn:3} and \eqref{eqn:2}, we have
\[
1 = \frac{s(1)d'(1)}{\beta}+  \ldots + \frac{s(n)d'(n)}{\beta^n} + \frac{s(1)d'(1)}{\beta^{n+1}}+ \ldots + \frac{s(n)d'(n)}{\beta^{2n}} + \frac{1}{\beta^{2n}}.
\]
Continuing this way, using \eqref{eqn:2}, we obtain that for any $m\geq1$,
\[
1= \sum_{j=0}^{m-1}\sum_{k=1}^n \frac{s(k)d'(k)}{\beta^{k+jn}} + \frac{1}{\beta^{mn}},
\]
and it follows that $1 = \sum_{j=1}^\infty s(j)d'(j)/\beta^j$. This expression is the \emph{infinite $(\beta,E)$-expansion of $1$}. Note that if $1$ has a finite $(\beta, E)$-expansion, then the corresponding infinite $(\beta, E)$-expansion is periodic. 

Now suppose that $x$ has a finite $(\beta, E)$-expansion. Then
\begin{align*}
x &= \frac{s(x, 1)d(x, 1)}{\beta}+ \frac{s(x, 2)d(x, 2)}{\beta^2} + \ldots + \frac{s(x, k)d(x, k)}{\beta^k} \\
&= \frac{s(x, 1)d(x, 1)}{\beta}+ \frac{s(x, 2)d(x, 2)}{\beta^2} + \ldots + \frac{s(x, k)d'(x, k)}{\beta^k} + \frac{1}{\beta^k},
\end{align*}
where we define $d'(x, k)=d(x, k)-1$ if $s(x, k)=1$, and $d'(x, k)= d(x, k)+1$ if $s(x, k)=-1$. 
Thus, 
the infinite $(\beta, E)$-expansion of $x$ is given by the sequence $vw$, where $$v= ((s(x, 1), d(x, 1)), \ldots, (s(x, k-1), d(x, k-1)), (s(x, k), d'(x, k))),$$ 
and $w$ is the (infinite) $(\beta, E)$-expansion of $1$. 
\subsection{Space of Itineraries} \label{sec:itinerary} There is another way to use $f$ to assign a sequence to a point: it is sometimes convenient to consider the itinerary  of a point relative to the partition $\{I_0, \ldots, I_m\}$, where $\beta \in (m, m+1]$. 

Let $\Sigma_m= \prod_{i=1}^\infty\{0, \ldots, m\}$. Given $x \in I$, its itinerary $$\It (x) =(\It(x,1), \It(x,2), \ldots)$$ under $f= f_{\beta, E}$ is the sequence in $\Sigma_m$ given by
\[
\It(x,j)=i \text{ if } f^{j-1}x \in I_i.
\]
Using the rules on the digits $d(x, i)$ and the signs $s(x, i)$, the $(\beta, E)$-expansion for $x$ can be recovered from $\It(x)$ and vice versa. In particular, we can map the $(\beta, E)$-expansion of $1$ to the itinerary of $1$ by the formula
\[
\It(1,j)=
\begin{cases}
 d(j) &\text{ if } s(j+1)=s(j) \\
d(j)-1 &\text{ if } s(j+1)=-s(j).
\end{cases}
\]

We recall the criteria of G\'ora for determining the validity of itineraries, and hence $(\beta, E)$-expansions. First we define an order $<_E$ on $\Sigma_m$. Given a finite word $w(1)\cdots w(j)$ from the alphabet $\{0, \ldots, m\}$, we let $\sign_E(w):= E(w(1))\cdots E(w(j))$.

We define the ordering $\leq_E$ by declaring $w <_Ev$ if $w(1)<v(1)$, or if $j$ is the first place where $w(j)\neq v(j)$, then
\[
w <_Ev \text{ if }
\begin{cases}
& w(j)< v(j) \text{ if } \sign_E(w(1) \cdots w(j-1)) = 1 \\
& w(j)> v(j) \text{ if } \sign_E(w(1) \cdots w(j-1)) = -1.
\end{cases}
\]
The order $\leq_E$ also makes sense on the set of finite sequences $\prod^k_{i=1}\{0, \ldots, m\}$ for any fixed $k\geq1$. Proposition 5 of G\'ora \cite{pG07} says that a sequence $w \in \Sigma_m$ is the itinerary of a point $x$ under $f_{\beta, E}$ if and only if for all $j\geq0$, $\sigma^jw \leq_E \It(1)$. We remark that by taking the closure of the space of all such itineraries in $\Sigma_m$, this criteria can be thought of as determining the symbolic dynamics associated to $f_{\beta, E}$.

The order $\leq_E$ is an essential ingredient in the theory of one-dimensional maps, and has its roots in the work of Parry \cite{wP64b}. This is a special case of the characterization of symbolic dynamics of piecewise monotonic maps that is formulated more generally in e.g. \cite{FP09}. Similar ideas appear in the celebrated work of Milnor and Thurston \cite{MT} for continuous multimodal maps, where they assign to each point a sequence $\theta(x)$, called the invariant coordinate of $x$. For points $x$ that are not pre-images of a critical point, the sequence $\theta(x)$ is exactly determined by the itinerary and sign data $(s(x,i))_{i \in \NN}$ of $x$.


\subsection{Key identities for generalized $\beta$-transformations}  First, we establish the fundamental relationship between the coefficients $d(j)$, $s(j)$, and the `signed orbit' of $1$ which we write $c_j:= s(j+1)f^j(1)$. Note that $c_0=1$.
\begin{lemma} For $j\geq0$, the coefficients satisfy the recursion relation
\begin{equation} \label{recurs}
\beta c_j -s(j+1) d(j+1) = c_{j+1}.
\end{equation}
\end{lemma}
\begin{proof} First, we rewrite the map $f$ as
\[
f(x) = e(x,1)(\beta x - d(x,1)),
\]
and thus
\[
f^j(x) = e(x, j)(\beta f^{j-1}(x) - d(x, j)).
\]
In particular, we have $f^j(1) = e(1, j)(\beta f^{j-1}(1) - d(j))$ for $j\geq1$. Thus, for $j\geq0$, we have
\[
\beta f^j(1) - d(j+1) = e(1, j+1) f^{j+1}(1).
\]
Multiplying by $s(j+1)$ yields
\[
\beta s(j+1) f^j(1) - s(j+1)d(j+1) = s(j+1) e(1, j+1) f^{j+1}(1) = s(j+2) f^{j+1}(1),
\]
which establishes \eqref{recurs}.
\end{proof}
Now we prove an identity which is key to our analysis, generalizing an identity which was observed in Solomyak \cite{So94} for $\beta$-transformations.
\begin{lemma} \label{identity0}
For any $z$ with $|z|>1$,
\[
1 - \sum_{j=1}^\infty s(j)d(j) z^{-j} = \left(1- \frac{\beta}{z} \right )\sum_{j=0}^{\infty} c_j z^{-j},
\]
where $c_j =s(j+1) f^{j}(1)$.
\end{lemma}
\begin{proof}
We assume that $|z|>1$ so that the above series converge. Multiplying out the right hand side, we obtain
\[
\left (1- \frac{\beta}{z} \right)\sum_{i=0}^{\infty} c_i z^{-i}= 1+(c_1-\beta)z^{-1}+ \ldots + (c_{j+1}-c_j\beta)z^{-(j+1)} + \ldots.
\]
For all $j\geq 0$, we have $c_{j+1}-c_j\beta = -s(j+1)d(j+1)$ by \eqref{recurs}, which yields the required inequality.
\end{proof}
\subsection{Post-critically finite generalized $\beta$-transformations} \label{sec:pcf}
We define a generalized $\beta$-transformation to be \emph{post-critically finite (PCF)} if the orbit of $1$ is finite, i.e. $\{f^j(1) \mid j\geq0\}$ takes finitely many values. More generally, we say a piecewise monotonic map (not necessarily continuous) is \emph{post-critically finite} if all maxima, minima and discontinuity points have a finite orbit. For a generalized $\beta$-transformation, all discontinuity points are pre-images of the points $1$ or $0$, and $0$ is either a fixed point or satisfies $f(0)=1$, so these definitions agree. We choose this terminology in order to be consistent with the literature on continuous multimodal maps, where PCF is the standard term for a map whose topological critical points have a finite orbit.

By the general theory of piecewise monotonic maps \cite{MT, ALM}, if the map is PCF, then it admits a Markov partition. That is, there is a partition $\mathcal P$ of the interval into subintervals such that for all $P \in \mathcal P$, $\overline{f(P)}$ is the closure of a union of elements of $\mathcal P$. The partition $\mathcal P$ is obtained by taking subintervals whose endpoints are the forward orbits of the critical points.   Thus, PCF interval maps are the ones that can be modeled by a shift of finite type, and thus have a well understood orbit structure. This motivates why we investigate which interval maps are PCF. 

\section{Generalized Parry numbers and Parry polynomials} \label{sec3}
We review the definition of a Parry number, and a Parry polynomial from the $\beta$-transformation literature, and extend these concepts to generalized $\beta$-transformations. Parry numbers and the Parry polynomial were both introduced in his seminal paper on $\beta$-expansions \cite{Pa}. 

\subsection{Parry numbers and Parry polynomials}
A number $\beta>1$ is a \emph{Parry number} if the $\beta$-expansion of $1$ is pre-periodic. This occurs if and only if $1$ has a finite orbit under $f_\beta$. We say that a Parry number is a \emph{simple Parry number} if the $\beta$-expansion of $1$ is periodic. 

For a Parry number, the \emph{Parry polynomial} $P_\beta(z)$ is a monic polynomial with integer coefficients which is naturally associated to the infinite $\beta$-expansion of $1$.  We obtain $P_\beta$ by taking the infinite $\beta$-expansion of 1
\[
1=\sum_{j=1}^{\infty} \frac{d(j)}{\beta^j},
\] 
and using the geometric series formula on the right hand side. We multiply through so all $\beta$ have a non-negative exponent, and bring all terms to one side. The resulting expression is the formula $P_\beta(\beta) =0$. For a simple Parry number with infinite $\beta$-expansion of $1$ given by $(d(1), \ldots, d(p))^\infty$, we arrive at the expression
\begin{align*}
P_\beta (z) &= z^p- \sum_{j=1}^p d(j) z^{p-j} -1 \\
& = z^p - d(1)z^{p-1} - d(2) z^{p-2} - \cdots - d(p-1) z -1 - d(p).
\end{align*}

Since $P_\beta(\beta) =0$, all Galois conjugates of $\beta$ must also satisfy $P_\beta (z) =0$.  The polynomial $P_\beta$ is not necessarily irreducible (i.e. it might have higher degree than the minimal polynomial for $\beta$), so $P_\beta$ may have zeros which are not Galois conjugates of $\beta$.  Such a zero is called a \emph{$\beta$-conjugate}. The distribution of $\beta$-conjugates was studied in \cite{VG, VG2}. 




\subsection{Generalized Parry numbers and Parry polynomials} We define a number $\beta>1$ to be a \emph{generalized Parry number} if we can find $E$ so that the orbit of $1$ under the $(\beta, E)$-transformation is finite.  
Thus, in the terminology of \S\ref{sec:pcf}, $\beta$ is a generalized Parry number iff there exists $E$ so that the $(\beta, E)$-transformation is PCF. 


Let $f = f_{\beta,E}$ be a PCF generalized $\beta$-transformation. Since the orbit of $1$ is periodic or pre-periodic, then so is the sequence given by the infinite $(\beta, E)$-expansion of $1$.

\begin{definition}
For a post-critically finite $(\beta, E)$-transformation, let us write the infinite $(\beta, E)$-expansion of $1$ as $vw^\infty$, where
\[
v = ((s(1), d(1)), (s(2), d(2)), \ldots, (s(k), d(k))),
\]
\[
w =  ((s(k+1), d(k+1)), \ldots, (s(k+p), d(k+p))).
\]
In the above, $w$ is written with the lowest possible period, and in the periodic case $k=0$, $v$ is the empty word. We define the \emph{generalized Parry polynomial} to be
\[
P_{\beta, E}(z) = z^{k+p} - \sum_{j=1}^{k+p}s(j)d(j)z^{k+p-j} - z^k+\sum_{j=1}^ks(j)d(j)z^{k-j}.
\]
\end{definition}
The formula simplifies if the $(\beta,E)$-expansion of $1$ is periodic
, in which case
\[
P_{\beta, E} (z) 
 = z^p- \sum_{j=1}^p s(j)d(j) z^{p-j} -1. 
\]
The expression $P_{\beta, E}(\beta) =0$ can be derived from applying the geometric series formula to the $(\beta, E)$-expansion of $1$, which motivates the definition of $P_{\beta, E}$. The following lemma is based on this relationship. 
\begin{lemma} \label{keyidentity1}
A number $z$ with $|z| >1$ is a zero of $P_{\beta, E} (z)$  if and only if
\begin{equation} \label{identity}
1 - \sum_{j=1}^\infty s(j)d(j) z^{-j} =0.
\end{equation}
\end{lemma}
\begin{proof}
We first assume that the infinite $(\beta, E)$-expansion of $1$ is periodic. Suppose $z$ with $|z|>1$ satisfies $P_{\beta, E} (z) =0$.  Then
\[
1- z^{-p} = \sum_{j=1}^p s(j)d(j) z^{-j},
\]
and thus
\[
1 = \frac{\alpha}{1-z^{-p}},
\]
where $\alpha =  \sum_{j=1}^p s(j)d(j) z^{-j}$. Using the geometric series formula yields
\[
1 = \sum_{m=0}^\infty \alpha z^{-pm} =\sum_{m=0}^\infty \sum_{j=1}^p s(j)d(j) z^{-pm-j} = \sum_{j=1}^\infty s(j)d(j) z^{-j}.
\]
The general case follows the same strategy. Suppose now that the  infinite $(\beta, E)$-expansion of $1$ is pre-periodic, and $z$ with $|z|>1$ satisfies $P_{\beta, E} (z) =0$. Then
\[
z^{k+p}-\sum_{j=1}^k s(j)d(j)z^{k+p-j} -  z^k+\sum_{j=1}^ks(j)d(j)z^{k-j} = \sum_{j=k+1}^{k+p} s(j)d(j)z^{k+p-j} 
\]
Let $\alpha_1 = \sum_{j=1}^ks(j)d(j)z^{-j}$ and $\alpha_2 =  \sum_{j=k+1}^{k+p}s(j)d(j)z^{-j}$. Then we have
\[
(z^{k+p}-z^k)(1-\alpha_1) = \alpha_2z^{k+p},
\]
and thus \[ 1-\alpha_1 = \frac{\alpha_2}{1-z^{-p}}.\]
Using the geometric series formula yields
\[
1 = \alpha_1 + \sum_{m=0}^\infty \alpha_2 z^{-pm}= \sum_{j=1}^ks(j)d(j)z^{-j} + \sum_{m=0}^\infty\sum_{j=k+1}^{k+p} s(j)d(j) z^{-pm-j},
\]
and the right hand side is $\sum_{j=1}^\infty s(j)d(j) z^{-j}$, showing \eqref{identity}. The argument can be reversed to show the opposite implication.
\end{proof}
Setting $z=\beta$, the expression for the $(\beta,E)$-expansion of $1$ shows that $P_{\beta, E} (\beta) = 0$, and it follows that all Galois conjugates of $\beta$ with $|z|>1$ satisfy
\eqref{identity}.
\begin{remark}
We do not know whether $P_{\beta, E}$ is irreducible, so there may be $z$ which satisfy $P_{\beta, E} (z)=0$ but are not Galois conjugates of $\beta$. We call such $z$ the \emph{generalized $\beta$-conjugates}, following terminology of Verger-Gaugry in the $\beta$-transformation case \cite{VG, VG2}. The generalized $\beta$-conjugates also satisfy the equation \eqref{identity}. It would be interesting to extend the results of \cite{VG, VG2} to study the distribution of the generalized $\beta$-conjugates. 
\end{remark}

\begin{remark}
If the $(\beta, E)$-expansion of $1$ is finite, the polynomial $P_{\beta, E}$ can be obtained equivalently by looking directly at this finite $(\beta, E)$-expansion;  that is, the expression $1 = \frac{s(1)d(1)}{\beta}+ \frac{s(2)d(2)}{\beta^2} + \ldots + \frac{s(n)d(n)}{\beta^n}$. In this case, $P_{\beta, E}= z^n-\sum_{j=1}^ns(j)d(j)z^{n-j} = z^n-\sum_{j=1}^ns(j)d'(j)z^{n-j}-1.$ 
\end{remark}
\begin{remark} 
Liao and Steiner studied the class of negative $\beta$-transformations in \cite{LS}. This is the subclass of generalized $\beta$-transformations where the sign of all the branches is set to $-1$. They call a number $\beta$ for which the negative $\beta$-expansion of $1$ is pre-periodic a \emph{Yrrap} number, and give examples. 
It is immediate that every Yrrap number is a generalized Parry number.
\end{remark}

\subsection{Upper bounds on conjugates for PCF $(\beta, E)$-transformations}
Combining Lemmas \ref{identity0} and \ref{keyidentity1}, we see that any Galois conjugate of a generalized Parry number $\beta$  with $|z|>1$ satisfies
\begin{equation} \label{wtf}
1+ \sum_{j=1}^{\infty} c_j z^{-j} = \sum_{j=0}^{\infty} c_j z^{-j}  = 0,
\end{equation}
where $c_j =s(j+1) f^{j}(1) \in [-1, 1]$ and $f$ is a PCF $(\beta, E)$-transformation. (The same is true for the generalized $\beta$-conjugates with $|z|>1$). Consider the class of functions
\[
\FFF = \{ T(w) = 1+ \sum_{j=1}^{\infty} a_j w^{j} : a_j \in [-1,1]\}.
\]
If $\lambda$ is a zero of a function in $\mathcal F$, then $z= \lambda^{-1}$ satisfies
\[
1+ \sum_{j=1}^{\infty} a_j z^{-j} = 0.
\]
Thus  if $z$ is a Galois conjugate of a generalized Parry number $\beta$ with $|z| >1$, then $z^{-1}$ is a zero of the 
function in $\mathcal F$ with coefficients as in \eqref{wtf}. 
\begin{lemma} \label{zeroes}
Any zero $\lambda$ of any function in $\mathcal F$ has modulus at least $\frac{1}{2}$. If any of the $a_j$ satisfy $|a_j|<1$, then $|\lambda| > \frac{1}{2}$. 
\end{lemma}
\begin{proof}
Suppose $T(\lambda) = 1+ \sum_{j=1}^{\infty} a_j \lambda^{j} =0$, where $a_j \in [-1,1]$. We argue by contradiction.
Suppose that $|\lambda| < \frac{1}{2}$. Then $|a_j\lambda^j| = |a_j| |\lambda|^j < 2^{-j}$. Thus
\[
|\sum_{j=1}^{\infty} a_j \lambda^{j}| \leq \sum_{j=1}^{\infty} |a_j \lambda^{j}| < \sum_{j=1}^{\infty} 2^{-j} = 1,
\]
which contradicts the fact that if $\lambda$ is a zero of $T$, then $|\sum_{j=1}^{\infty} a_j \lambda^{j}| = 1$.
 If we assume further that at least one of the $a_j$ satisfy $|a_j|<1$, then for every $\lambda$ with $|\lambda|\leq \frac12$ we have $|\sum_{j=1}^{\infty} a_j \lambda^{j}|<1$, so the same contradiction argument applies.
\end{proof}
 Clearly, this  bound is sharp by setting all the $a_j=-1$ and letting $\lambda=\frac{1}{2}$.

\begin{theorem} \label{upper}
If $\beta$ is a generalized Parry number, then all Galois conjugates $z$ of $\beta$ satisfy $|z|<2$. 
\end{theorem}
\begin{proof}
By Lemmas \ref{identity0} and \ref{keyidentity1}, if $z$ is a Galois conjugate of $\beta$ with $|z|>1$, then $z^{-1}$ is a zero of a function in $\FFF$ whose coefficients are given by $a_j =s(j+1) f^{j}(1)$ for the appropriate $(\beta, E)$-transformation $f$. Thus we can apply Lemma \ref{zeroes} to show that $|z|< 2$. The inequality is strict because the only way we can have $|a_j|=1$ for all $j$ is if $1$ is a fixed point of $f$. This can only happen if $\beta$ is an integer, and thus does not have Galois conjugates.
\end{proof}
\begin{remark}
For $\beta$-transformations, Solomyak showed \cite{So94} that if $z$ is a Galois conjugate of $\beta$ with $|z| >1$, then $z^{-1}$ is a zero of a function in $\FFF_{[0,1]}$, where
\[
\FFF_{[0,1]} = \{ T(w) = 1+ \sum_{j=1}^{\infty} a_j w^{j} : a_j \in [0,1]\}.
\]
He used this characterization to show that $|z| \leq (\sqrt5+1)/2$ by exploiting the fact that $|a_j- \frac12| \leq \frac12$.
\end{remark}
\section{Conjugates with a prescribed argument and Solomyak's structure theorem} \label{sec4}
We investigate the maximum possible modulus of a conjugate with a prescribed argument. Let 
\[
\lambda_\phi = \min \{|\lambda| : \lambda \mbox{ is a zero of a function in }\mathcal F \mbox{ and the argument of } \lambda \mbox{ is } \phi\}. 
\]

A description of the function in $\FFF$ that attains $\lambda_\phi$ is the content of the Solomyak Structure Theorem. 
Such a function is called \emph{$\phi$-optimal}. This result was originally established by Solomyak \cite{So94} to analyze the zeroes of functions in $\FFF_{[0,1]}$.  The Structure Theorem was  generalized by Beaucoup, Borwein, Boyd and Pinner \cite{BBBP} to a class of power series with restricted coefficients that includes $\FFF$. The following statement is given in \cite{BBBP}.
\begin{theorem} [Solomyak Structure Theorem for $\FFF$]
Given an argument $\phi \in (0, \pi)$, there exists $\alpha = \alpha_\phi \in (0, \pi)$ and a function
\[
T_\phi(w) = 1+ \sum_{n=1}^{\infty} a_n w^{n}
\]
whose coefficients satisfy
\[
a_n = \begin{cases} 1 &  \mbox{ if } n\phi - \alpha \in (0, \pi) \pmod {2 \pi} \\
-1 &  \mbox{ if }  n\phi - \alpha \in (-\pi, 0) \pmod {2 \pi},
\end{cases}
\]
such that $T_\phi$ is $\phi$-optimal
; i.e. $T_\phi(\lambda_\phi e^{i\phi})=0$.
\end{theorem}
The sequence of coefficients $\{a_n\}$ is determined by the rotation sequence of slope $\phi$ and base point $\alpha$ with the possible exception of one coefficient $a_j$, which we call the \emph{anomalous coefficient}.  In the case of $\FFF_{[0,1]}$, where $\phi$-optimal functions have non-anomalous coefficients belonging to $\{0,1\}$, Solomyak gives explicit examples of $\phi$-optimal functions  for which the anomalous coefficient is different from all the rest.  The function $T_\phi$ is unique when $\phi/2\pi$ is irrational, and it is conjectured that when $\phi/2\pi$ is rational, $\phi$-optimal functions are never unique, see \cite{So94}.

No convenient characterizations of the anomalous coefficient or the function $\phi \to \alpha_\phi$ are currently available. Despite these drawbacks, Solomyak put his structure theorem to impressive use in \cite{So94}, obtaining results on the continuity and differentiability of the function $\phi \to \lambda_\phi$. In particular, we have the following result whose proof was given in the $\FFF_{[0,1]}$ case in \cite[Lemma 4.2]{So94}, and was observed to extend almost verbatim to $\FFF$ in \cite[Proposition 1]{BBBP}.
\begin{theorem} [\cite{So94, BBBP}] 
The function $\phi \to \lambda_\phi$ is continuous on $(0, \pi)$.
\end{theorem}

We apply this result in the following theorem.
\begin{theorem} \label{thm:bound}
 The quantity $\lambda_\phi$ is bounded uniformly away from $\frac 1 2$ for $\phi \in (0, \pi)$. Thus,  $\sup\{\lambda_\phi^{-1} : \phi \in (0, \pi)\} <2$.
\end{theorem}
\begin{proof}
First, we show that $\lambda_\phi >\frac{1}{2}$  on $(0, \pi)$. Suppose not. Then there exists $\theta \in (0, \pi)$, and $\lambda = \frac12 e^{i\theta}$, and $a_n \in [-1,1]$ so that
\[
1 + \sum a_n \lambda^n =0.
\]
Furthermore, by Lemma \ref{zeroes}, $a_n \in \{1, -1\}$ for all $n$. In particular, writing $$a_n \lambda^n= \frac{1}{2^n} e^{i\theta_n},$$
we have 
\[
1+ \sum_{n=1}^\infty \frac{1}{2^n} e^{i \theta_n} =0.
\]
Thus
\[
1+ \frac{1}{2} e^{i\theta_1} = - \sum_{n=2}^\infty \frac{1}{2^n}  e^{i \theta_n}.
\]
Letting $z_1 = 1+ \frac{1}{2} e^{i\theta_1}$. we see that $|z_1-1| = \frac{1}{2}$. Let $z_2=  -\sum_{n=2}^\infty \frac{1}{2^n}  e^{i \theta_n}$. Then $|z_2| \leq \sum_{n=2}^\infty \frac{1}{2^n} = \frac{1}{2}$. Thus, if $z_1=z_2$, then $z_1= \frac{1}{2}$, and so $\theta_1 = \pi$. Thus we have $a_1\lambda = -\frac{1}{2}$.
Since $a_1 \in \{1, -1\}$, it follows that $\lambda = \frac12$ or $\lambda = -\frac12$.
This contradicts the hypothesis that $\lambda = \frac12 e^{i\theta}$ for $\theta \in (0, \pi)$.

Since the map $\phi \to \lambda_\phi$ is continuous on $(0, \pi)$, all that remains is to analyze the limit of $\lambda_\phi$ as $\phi \to 0$ and $\phi \to \pi$. The argument of \cite[Lemma 4.1]{So94} shows that $\lim_{\phi \to 0} \lambda_\phi$ converges to a value which is a double root of some function in $\mathcal F$, and likewise for $\lim_{\phi \to \pi} \lambda_\phi$. Solomyak gives a rigorous argument that the smallest double root of a function in $\mathcal F$ is greater than 0.6299 on p.622 of \cite{So95}. Thus $\min \{\lim_{\phi \to \pi} \lambda_\phi, \lim_{\phi \to 0} \lambda_\phi \}\geq 0.6299 >\frac{1}{2}$. It follows that $\sup\{\lambda_\phi^{-1} : \phi \in (0, \pi)\} <2$.
\end{proof}
\begin{remark}
Note that $\lambda_0 = \lambda_\pi = \frac{1}{2}$. This is because $\lambda = \frac12$ is a zero of the function in $\FFF$ with $a_n = -1$ for all $n$, and $\lambda = - \frac 12$ is a zero of the function in $\FFF$ with $a_n = (-1)^{n+1}$ for all $n$. Thus the function $\phi \to \lambda_\phi$ is discontinuous at $0$ and $\pi$.
\end{remark}
\begin{remark}
The function $\phi \to \lambda_\phi$  is investigated numerically in \cite{BBBP}, and is plotted as figure 2 of \cite{BBBP}. The numerics suggest that $\lambda_\phi \in(0.63, 0.71)$ for $\phi \in (0, \pi)$, and that $\lim_{\phi \to \pi} \lambda_\phi \approx 0.6491$. Thus, we should expect that $\sup\{\lambda_\phi^{-1} \mid \phi \in (0, \pi)\}<1.59$. It is an open question to rigorously establish optimal bounds on this quantity.
\end{remark}

\begin{remark}
Solomyak studied differentiability properties of the boundary of $\FFF_{[0,1]}$ in some detail using the tools introduced above, giving criterion for when $\phi \to \lambda_\phi$ is smooth, and non-smooth. The extension of these results to $\FFF$ is claimed in \cite[Proposition 1]{BBBP}, with proof referred to Solomyak. 
\end{remark}


\section{The set $\overline \Omega$} \label{sec3.1}
Recall that $\Omega$ is defined to be the set of all Galois conjugates for all generalized Parry numbers. Let $\mathcal G$ denote the set of zeroes of functions in $\FFF$. We have shown that 
$\{ z \in \Omega \mid |z| >1 \} \subset \{ z \mid z^{-1} \in \mathcal G \}$. In this section, our main result is a partial converse of this: we will show that  $\{ z \mid z^{-1} \in \mathcal G \} \subset \{ z \in \overline \Omega \mid |z| >1 \} $. 
To prove this, first we require  a means of identifying when a number is a generalized Parry number. 
\subsection{Criteria for $\beta$ to be a generalized Parry number}
Suppose that $((s(1), a(1)), \ldots, (s(n), a(n)))$ is a finite sequence where $n\geq 2$, the $a(j)$ are non-negative integers with $a(n) \neq 0$, the $s(j) \in\{1,-1\}$, and $s(1)=1$. 
Let the finite sequence $(\It(1), \ldots, \It(n))$ be given by
\[
\It(j)=
\begin{cases}
 a(j) &\text{ if } s(j+1)=s(j) \\
a(j)-1 &\text{ if } s(j+1)=-s(j).
\end{cases}
\]
for $j\in\{1, \ldots, n-1\}$, and $\It(n)=a(n)$.
We impose the following hypotheses on the sequence:
\begin{enumerate}
\item $\It(1) > a(j)$ for all $j\geq2$;
\item there exists a sign configuration $E \in \{1,-1\}^{m+1}$, where $m = \It(1)$, such that $s(j+1)= s(j)E(\It(j))$ for each $j\in\{1, \ldots, n-1\}$. \item if $a(j)=0$, then $s(j+1) = s(j)$, to ensure that $\It(j) \geq 0$. Thus, if any of the $a(j)$ are $0$, then $E(0)=1$.
\end{enumerate}
Writing $w= (\It(1), \ldots, \It(n))$, we will find a $\beta$ such that the itinerary of $1$ for $f_{\beta, E}$ is either $w^\infty$ or $w0^\infty$.

 A complete characterization of which sequences arise as the itinerary of $1$ for some $(\beta, E)$-transformation is currently open. A statement for $\It(1) \geq 2$ appears as Theorem 25 of G\'ora \cite{pG07}, although his hypotheses have been criticized by Steiner in \cite{St13}. 
We do not pursue the general case here since we are investigating only PCF transformations.

We define the function
$$F(x)= s(1)a(1)+\frac{s(2)a(2)}{x} + \frac{s(3)a(3)}{x^2} + \ldots + \frac{s(n)a(n)}{x^{n-1}}.$$ 
We want to show that $F$ has a fixed point in the interval $(\It(1), \It(1)+1)$. To this end, for $j \in \{1, \ldots, n-1\}$, let
\[
R_j(x) = \sum_{i=j}^{n-1} \frac{s(i+1)a(i+1)}{x^i}.
\]
\begin{lemma} \label{5.1}
For $x \in [\It(1), \It(1)+1]$, and $j \in \{1, \ldots, n-1\}$, we have 
\begin{enumerate}
\item $|R_j(x)| < \frac{1}{x^{j-1}}$
\item $\sign(R_j(x))= s(j+1)$.
\end{enumerate}
\end{lemma}
\begin{proof}
Let $N= \max\{a(2), \ldots, a(n)\}$. Then, since $N \leq \It(1)-1 \leq x-1$,
\[
|R_j(x)| < \sum_{i=j}^{\infty} \frac{N}{x^i} =\frac{1}{x^{j-1}} \frac{N}{x-1}\leq \frac{1}{x^{j-1}}.
\]
Now take the first $k\geq1$ so that $a(j+k)\neq0$. If $k \geq 2$, then since $a(j+1) = \cdots = a(j+k-1)=0$, we have $s(j+k)= s(j+1)$. We have
\[
R_j(x)= \frac{s(j+k)a(j+k)}{x^{j+k-1}} + R_{j+k}(x),
\]
and thus
\[
\frac{s(j+k)a(j+k)}{x^{j+k-1}} -\frac{1}{x^{j+k-1}}< R_j(x) < \frac{s(j+k)a(j+k)}{x^{j+k-1}}+\frac{1}{x^{j+k-1}}.
\]
Since $a(j+k)\geq1$, it follows that $\sign(R_j(x))=s(j+k)=s(j+1)$.
\end{proof}
It follows immediately from (2) in Lemma \ref{5.1} that 
\begin{equation} \label{useful0}
|R_j(x)| = s(j+1) R_j(x).
\end{equation}
\begin{lemma} \label{lem:criteria}
There exists $\beta \in(\It(1), \It(1)+1)$ such that
\[
\beta = s(1)a(1)+\frac{s(2)a(2)}{\beta} + \frac{s(3)a(3)}{\beta^2} + \ldots + \frac{s(n)a(n)}{\beta^{n-1}}
\]
\end{lemma}
\begin{proof}
We show that $F:[\It(1), \It(1)+1] \to [\It(1), \It(1)+1]$. Note that $F(x) = a(1) + R_1(x)$. There are two cases. First suppose $s(2)=1$. Then $E(\It(1))=1$ and $\It(1) = a(1)$. Since $0<R_1(x)<1$, we have
\[
\It(1)+1 = a(1)+1 > a(1) + R_1(x) > a(1)= \It(1). 
\]
Now suppose $s(2)=-1$. Then $E(\It(1))=-1$ and $\It(1) = a(1)-1$. Since $0>R_1(x)>-1$, we have
\[
\It(1)+1 = a(1) >  a(1) + R_1(x) > a(1) -1 = \It(1). 
\]

Thus, in both cases, the image of $F$ is contained in $(\It(1), \It(1)+1)$.  Considering the map as $F:[\It(1), \It(1)+1] \to [\It(1), \It(1)+1]$, it follows from the Intermediate Value Theorem that $F$ has a fixed point $\beta$. Clearly, $\beta \notin \{\It(1), \It(1)+1\}$.
\end{proof}
From now on, we fix $\beta \in(\It(1), \It(1)+1)$ provided by Lemma \ref{lem:criteria}, and let $f= f_{\beta, E}$. Recall that $[0,1]$ is partitioned into intervals $I_0 = \left[0, 1/\beta \right ]$, $I_1 = \left(1/\beta, 2/\beta \right ]$,$\ldots$, $I_m = \left(m/\beta, 1 \right ]$, where $m=\It(1)$.

 \begin{lemma} \label{useful1}
 For $k\in \{1, \ldots, n-1\}$, we have 
 \begin{enumerate}[(i)]
  \item $f^{k-1}(1)\in I_{\It(k)}$
 \item $f^{k-1}(1)= \frac{1}{\beta} \left(a(k) + \beta^{k-1}E(\It(k))|R_k(\beta)| \right)$

 \end{enumerate}
 \end{lemma}
 \begin{proof}
We argue recursively. For $k=1$, it is immediate that $1\in I_{\It(1)}$, and we have
 \begin{align*}
1 &= \frac{s(1)a(1)}{\beta}+ \frac{s(2)a(2)}{\beta^2} + \ldots + \frac{s(n)a(n)}{\beta^n}\\
&= \frac{a(1)}{\beta}+ \frac{1}{\beta}R_1(\beta),
\end{align*}
and since $R_1(\beta) = s(2)|R_1(\beta)|=E(\It(1))|R_1(\beta)|$, we are done. 

Now we show that if $k \in \{1, \ldots, n-1\}$, and $(i)$ and $(ii)$ hold true for $f^{k-1}(1)$, then
\begin{equation} \label{useful}
f^k(1) =  \beta^{k-1}|R_k(\beta)|.
\end{equation}
There are two cases:

 Case (a): Suppose that $s(k+1)=s(k)$. Then $a(k)=\It(k)$, and $E(\It(k))=1$.  Since $E(\It(k))=1$, for $y\in I_{\It(k)}$, $f(y)= \beta y- a(k)$, and thus applying $f$ to $f^{k-1}(1) \in  I_{\It(k)}$, it follows from the expression $(ii)$ that $f^k(1) = \beta^{k-1}|R_k(\beta)|$.

Case (b): Suppose that $s(k+1)=-s(k)$. Then $a(k)-1=\It(k)$, and $E(\It(k))=-1$. Since $E(\It(k))=-1$, for $y\in I_{\It(k)}$, $f(y)=a(k) -\beta y$. Since $f^{k-1}(1) \in  I_{\It(k)}$, it follows from the expression $(ii)$ that $
f^k(1) = - E(\It(k))\beta^{k-1}|R_k(\beta)| =  \beta^{k-1}|R_k(\beta)|.$

Now fix $k\in \{1, \ldots, n-2\}$, and suppose that $(i)$ and $(ii)$ hold true for $f^{k-1}(1)$. It follows from \eqref{useful0} and \eqref{useful} that
\begin{align*}
f^k(1) &= \beta^{k-1}|R_k(\beta)| = \beta^{k-1}s(k+1)R_k(\beta) \\ 
&= \beta^{k-1}s(k+1)\left( \frac{s(k+1)a(k+1)}{{\beta}^k} + R_{k+1}(\beta) \right) \\
&= \frac{1}{\beta}\left (a(k+1)+\beta^ks(k+1)R_{k+1}(\beta) \right )\\
&= \frac{1}{\beta}\left (a(k+1)+\beta^ks(k+1)s(k+2)|R_{k+1}(\beta)| \right ).
\end{align*}
Since $s(k+1)s(k+2)= s(k+1)^2E(\It(k+1))$, we have established the formula $(ii)$ for $f^k(1)$. 

Now, we show that $(ii)$ implies $(i)$. Again, there are two cases. If $s(k+2)=s(k+1)$, then $a(k+1)= \It(k+1)$, and $E(\It(k+1))=1$, so
\[
\frac{a(k+1)}{\beta} < f^{k}(1)= \frac{1}{\beta}\left (a(k+1) + \beta^{k}|R_{k+1}(\beta)| \right )< \frac{a(k+1)+1}{\beta},
\]
and so $f^{k}(1) \in I_{\It(k+1)}$.

In the case that $s(k+2)=-s(k+1)$, then $a(k+1)-1= \It(k+1)$, and $E(\It(k+1))=-1$, so
\[
\frac{a(k+1)-1}{\beta} < f^{k}(1)= \frac{1}{\beta}\left (a(k+1) - \beta^{k}|R_{k+1}(\beta)| \right )< \frac{a(k+1)}{\beta},
\]
and so $f^{k}(1) \in I_{\It(k+1)}$. 

This shows that both $(i)$ and $(ii)$ are true for $f^k(1)$, which completes the proof.
 \end{proof}
 \begin{lemma} \label{Parry}
 We have $f^{n-1}(1)=a(n)/\beta$. Thus $f$ is PCF, and so $\beta$ is a generalized Parry number.
 \end{lemma}
 \begin{proof}
Since by Lemma \ref{useful1}, $(i)$ and $(ii)$ hold for $f^{n-2}(1)$, the equation \eqref{useful} shows that $f^{n-1}(1)=\beta^{n-2}|R_{n-1}(\beta)|= a(n)/\beta$. Since $f(a(n)/\beta)$ is either $1$ or $0$ (noting that $a(n)/\beta \in I_{a(n)-1}$, so this depends only on whether $E(a(n)-1)$ is positive or negative), we have shown that $1$ has a finite orbit, and hence $f$ is PCF.
 \end{proof}
The previous results allow us to give the following simple criterion for finding generalized Parry numbers.
 



\begin{theorem} \label{thm:criterion}
Suppose that $M(1), \ldots M(n)$ are distinct non-zero integers such that $M(1)\geq 2$, $|M(j)|+1 <M(1)$ for all $j\geq2$, and $|M(j)|\neq |M(k)|-1$ for all $j, k$.  Then the equation
\begin{equation} \label{eqn2}
x = M(1)+\frac{M(2)}{x} + \frac{M(3)}{x^2} + \ldots + \frac{M(n)}{x^{n-1}}
\end{equation}
has a solution $\beta>1$ which is a generalized Parry number. 
\end{theorem}
\begin{proof}
Let $a(j)= |M(j)| \geq 1$ and $s(j) = \sign(M(j))$, and define $\It(j)$ as we did earlier in this section. Our hypotheses imply that $\It(1)> a(j)$ for all $j\geq2$. Since all the $\It(j)$ are distinct, we have freedom to choose a vector $E$ whose entries in the positions $E(\It(j))$ leads to the sequence of signs $s(j)$.  Thus applying Lemma \ref{lem:criteria}, \eqref{eqn2} has a solution $\beta \in (\It(1), \It(1)+1)$, and so by Lemma \ref{Parry}, $f_{\beta, E}$ is PCF, and thus $\beta$ is a generalized Parry number.
\end{proof}
\begin{remark}
We mention some other classes of numbers that are known to be Parry numbers, and hence generalized Parry numbers. Schmidt proved that if $\beta$ is a Pisot number, 
 then $\beta$ is Parry \cite{kS80}.  Whether Salem numbers are Parry is a challenging open problem \cite{Boy, wT14}, first raised by Schmidt in \cite{kS80}.  Boyd proved that degree 4 Salem numbers are Parry \cite{dB89}. Numerical evidence and hueristic arguments by Thurston \cite{wT14} and Boyd suggest that most higher degree Salem numbers are not Parry, but it seems to be very difficult to find even a single rigorous example of this phenomenon.
\end{remark}

\subsection{Characterization of $\overline \Omega$}  
We now prove our main result about $\overline \Omega$.
\begin{theorem} \label{thm:closure}
The set $\overline \Omega$ is $\mathbb D \cup \{ z \mid |z| >1 \text{ and } z^{-1} \in \mathcal G \}$, where $\mathbb D =\{z \mid |z| \leq 1\}$.
\end{theorem}
\begin{proof}
We already know that the closed unit disk $\mathbb D$ is a subset of $\overline \Omega$, because $\mathbb D$ is contained in the closure of the Galois conjugates of the simple Parry numbers (see Theorem 2.1 of \cite{So94}). Thus, all it remains to show is that if $|z|>1$ and $z^{-1} \in \GGG$, then $z \in \overline \Omega$. The argument is a generalization of the second half of the proof of \cite[Theorem 2.1]{So94}.

Let $\lambda = z^{-1}$. Then $T(\lambda)=0$ for some $T(w) = 1+ \sum_{j=1}^\infty a_jw^j$ with $a_j \in [-1,1]$. We approximate $T$ with a function $g(w) = 1+ \sum_{j=1}^{n-1} b_jw^j$ with $b_j \in (-1,1)$, where all $b_j$ are rational. By taking $n$ large, and the $b_j$ arbitrarily close to the $a_j$, we can ensure that $g$ has a zero arbitrarily close to $\lambda$. 
Writing each $b_j$ in the form $b_j = M(j+1)/M(1)$, where $M(n) \in \ZZ$, we have 
\[
g(w) = M(1)^{-1}(M(1) + M(2)w + \ldots + M(n) w^{n-1}).
\]
We can make sure that our $b_j$ are chosen so that all $M(j)$ are distinct and non-zero, $M(1) \geq 2$, $|M(j)|+1 <M(1)$ for all $j$, and $|M(j)|\neq |M(k)|-1$ for all $j, k$. Now take a  large prime $p>M(n)$ (to be fixed later). We know by Theorem \ref{thm:criterion} that the equation
\[
x = pM(1)+\frac{pM(2)}{x} + \frac{pM(3)}{x^2} + \ldots + \frac{pM(n)}{x^{n-1}}
\]
has a solution $\beta>1$ which is a generalized Parry number. Define the polynomial
\[
q(w)= w^n - pM(1)w^{n-1}- pM(2)w^{n-2} - \ldots - pM(n),
\]
so that $\beta$ is a zero of $q$. The polynomial q is irreducible by Eisenstein's Criterion, and hence $q$ is the minimal polynomial for $\beta$. Thus all other zeroes of $q$ are conjugates of $\beta$. 

The rest of the argument is to show that one of these zeroes is close to $z$. Elementary computation shows that $q(w^{-1})=0$ if and only if $g(w)=\frac{1}{wpM(1)}$. Thus letting $h(w)= -\frac{1}{pM(1)w}$, we have  $q(w^{-1})=0$ if and only if $(g+h)(w)=0$. We now use Rouch\'e's theorem. We are free to choose $p$ as large as we like, so we can ensure that $|h|<|g|$ on a small circle $\gamma$ centered at $\lambda$. Thus, $g+h$ has the same number of zeroes as $g$ inside $\gamma$. By the choice of $g$, we can ensure that $g$, and hence $g+h$, has a zero $w_0$ inside $\gamma$. Thus, $q$ has a zero $w_0^{-1}$ with $w_0$ in a neighbourhood of $\lambda$. We can ensure that $w_0$, which by construction is the inverse of a Galois conjugate of a simple generalized Parry number, is arbitrarily close to $\lambda$. This completes the proof. 
\end{proof}

\subsection{Path-connectedness of $\overline \Omega$} Given Theorem \ref{thm:closure}, it is now easy to show that $\overline \Omega$ is path connected.  
We use the following lemma.
\begin{lemma} \label{starconvex}
Let $z=re^{i\theta} \in \overline \Omega$ with $r>1$. Then for all $r' \in (1, r)$, $z'=r'e^{i\theta} \in \overline \Omega$.
\end{lemma}
\begin{proof}
By Theorem \ref{thm:closure}, we have $\lambda = z^{-1} \in \GGG$, so $T(\lambda)=0$ for some $T \in \FFF$. For any $a>1$, the function $\overline T(z):= T(z/a) \in \FFF$, and has $a \lambda$ as a zero. Thus $a\lambda \in \GGG$, and so $\frac 1 a z \in \overline \Omega$ for any $a>1$.
\end{proof}
Thus, we can connect any two points in $\overline \Omega$ using, for example, paths along at most two radial lines together with a path along the unit circle $S^1$.

\section{Unimodal maps as generalized $\beta$-transformations}\label{unimodal}
We now use the results of the previous sections to study topological entropy for PCF continuous unimodal maps. The topological entropy of a continuous map on a compact metric space is a number that captures the exponential growth rate of distinct orbits of length $n$, and is a fundamental invariant of a topological dynamical system. See Walters for a general definition \cite{pW82}. The problem of deciding if a number can be obtained as the entropy of a map from a given class of systems has a long history with notable results including \cite{Li84, HM, wT14}.

For piecewise monotonic interval maps, the topological entropy has a simple formula which was first obtained by Misiurewecz and Szlenk \cite{MS, ALM}, and which for current purposes we take as our definition.
\begin{definition}
Let $f$ be a piecewise monotonic map of the unit interval. The \emph{topological entropy}, which we denote $\htop(f)$, or simply $h$, is defined to be
\[
\htop(f) = \lim_{n \to \infty} \frac 1 n \log \#\{\text{ \emph{branches of monotonicity for }} f^n\}.
\]
\end{definition}
Note that it is immediate from the definition that the entropy of a unimodal map is at most $\log2$. One can prove that if $f$ is a generalized $\beta$-transformation, then $\htop(f)= \log \beta$, see e.g. Corollary 4.3.13 of \cite{ALM}. 
It is well known that every unimodal map $f$ is topologically semi-conjugate to a $\lambda$-uniform expander $g$; that is, a piecewise affine continuous interval map whose slope on each interval of monotonicity is either  $\lambda$ or $-\lambda$. In particular, $\htop(f)= \htop(g)$, and if $f$ is PCF, then $g$ is PCF. This result was first proved in \cite{MT}, and is given as Theorem 4.6.8 of \cite{ALM}. Thus, to study the entropies of PCF unimodal maps, it suffices to study the entropies of PCF uniform expanders. Unimodal uniform expanders (perhaps after modifying by a conjugacy) can be thought of as generalized $\beta$-transformations. Thus the formalism of generalized $\beta$-transformations can be used to study the entropy of unimodal maps.  In particular, we have the following lemma.
\begin{lemma} \label{lem:conjugate}
Every PCF unimodal map is conjugate to a PCF generalized $\beta$-transformation
\end{lemma}
\begin{proof}
It suffices to show that every PCF $\lambda$-uniform expander $g$ is conjugate to a PCF generalized $\beta$-transformation. First, we use the standard trick of trimming the domain of $g$, and rescaling to get a surjective map of the unit interval. Let $\Lambda = \bigcap g^i ([0,1]) = [a,b]$, and consider $g|_\Lambda$. The critical point $c$ satisfies either $g(c)=a$ or $g(c)=b$. We also have $a, b \in \{g(a), g(b), g(c)\}$. We can conjugate by the affine transformation $\pi(x ) = \frac{1}{b-a}(x-a)$ to the map $G (x)= \pi \circ g \circ \pi^{-1}(x) = \frac{1}{b-a}[g((b-a)x +a)-a]$. Clearly an affine transformation will send critical points to critical points, and a conjugacy sends periodic orbits to periodic orbits, so the new map $G$ is PCF, surjective, and has domain $[0,1]$.

To be surjective, $G$ must have at least one full branch. There are four possibilities:
\begin{enumerate}
\item first branch full; sign configuration $(1,-1)$;

\item first branch full; sign configuration $(-1,1)$;

\item second branch full; sign configuration $(-1,1)$;

\item second branch full; sign configuration $(1,-1)$.
\end{enumerate}

Cases (1) and (2) are generalized $\beta$-transformations, with $\beta = \lambda$, and the appropriate sign configuration. For case (3), we conjugate by the transfomation $\pi (x) = -x$ to get the transfomation $h(x) = -G(-x)$ defined on $[-1, 0]$. We can conjugate by a translation to return the domain to $[0,1]$. The new map is in case (2). Similarly, a map in case (4) is conjugate to a map in case (1). Thus, $G$ is either a generalized $\beta$-transformation, or conjugate to a generalized $\beta$-transformation by an affine transformation.
\end{proof}

Thus, for a PCF unimodal map $f$, we have $\htop(f) = \log \beta$ for some generalized Parry number $\beta \in[1,2]$. So we have
\[
\Omega_T :=\{z \mid z\text{ is a conjugate of }e^{h(f)} \text{ for a PCF unimodal map }f \} \subset \Omega.
\]
The problem of giving a  description of $\Omega_T$ was raised in Thurston's final paper \cite{wT14}. His numerical results showed that apart from a spike along the real axis, this set appears to lie in a disk much smaller than the disk $|z|<2$.  Our description of $\Omega$ allows us to conclude that $\Omega_T \setminus \RR$ 
indeed lies in a disk of radius strictly less than $2$. As mentioned previously, although numerical results suggest that we should expect a bound less than $1.59$, rigorous sharp bounds are currently out of reach. Nevertheless, we establish the principle that all non-real Galois conjugates are contained inside a disk with a smaller radius than the trivial bound $2$. We now state this as a theorem. 
The proof is an immediate consequence of the discussion above and Theorem \ref{thm:bound}.
\begin{theorem}
There exists $\epsilon >0$ so that if $z$ is a conjugate of $e^{h(f)}$ for a PCF unimodal map $f$ and $z \notin \RR$, then $|z|<2-\epsilon$.
\end{theorem}

Many questions remain about the sets $\Omega$ and $\Omega_T$. For example, can one describe $\overline \Omega \setminus \overline \Omega_T$? One way in which these sets differ is that, from Thurston's picture, $\overline \Omega_T$ appears to have `holes' around some roots of unity. This is ruled out for $\overline \Omega$ by the star-convexity proved in Lemma \ref{starconvex}. It would be interesting to determine the exact location and distribution of the holes that appear in $\overline \Omega_T$. The existence of holes for $\overline \Omega_T$ inside the unit disk was established in \cite{CKW}. Another question is to determine where, and by how much, the outer boundaries of $\overline \Omega$ and $ \overline \Omega_T$ differ. Numerical investigation of these questions could be a good first step towards rigorous results.

\bibliography{conjugatesforbeta}
\bibliographystyle{plain}

\end{document}